\documentclass[a4paper]{amsart}

\usepackage{amsthm,amsfonts,amsmath, amssymb}

\usepackage{cmbright}

\usepackage[T1]{fontenc}

\usepackage{anysize}
\marginsize{3cm}{3cm}{3cm}{3cm}

\renewcommand{\baselinestretch}{1.25}
\usepackage{graphicx,float,enumerate,subfigure,tikz}

\usepackage{verbatim}

\usepackage{url}

\usepackage[capitalise]{cleveref}

\newtheorem{theorem}{Theorem}[section]
\newtheorem{corollary}[theorem]{Corollary}
\newtheorem{lemma}[theorem]{Lemma}
\newtheorem{proposition}[theorem]{Proposition}
\newtheorem{problem}[theorem]{Problem}

\theoremstyle{definition}

\theoremstyle{remark}
\newtheorem{remark}[theorem]{Remark}

\usepackage{mathtools}

\DeclareMathOperator{\bipyr}{bipyr}
\DeclareMathOperator{\aff}{aff}
\DeclareMathOperator{\conv}{conv}

\DeclareMathOperator{\ver}{vert}

\DeclareMathOperator{\aof}{AOF}

\crefname{remark}{Remark}{Remarks}

\crefname{rmk}{Remark}{Remarks}
\crefname{problem}{Problem}{Problems}

\date{\today}
\title{On the  reconstruction of polytopes}

\author{Joseph Doolittle}
\address{College of Liberal Arts and Sciences, University of Kansas, Lawrence, KS, USA}
\email{\texttt{jdoolitt@ku.edu}}
\author{Eran Nevo}
\thanks{Research of E.~Nevo was partially supported by Israel Science Foundation grant ISF-1695/15.}
\address{Institute of Mathematics, Hebrew University of Jerusalem}
\email{\texttt{nevo@math.huji.ac.il}}
\author{Guillermo Pineda-Villavicencio}
\thanks{Research of Pineda-Villavicencio was partly supported by the Indonesian government Scheme P3MI, Grant No. 1016/I1.C01/PL/2017.}
\address{Centre for Informatics and Applied Optimisation, Federation University Australia}
\email{\texttt{work@guillermo.com.au}}
\author{Julien Ugon}
\thanks{Research of Ugon was supported by ARC discovery project DP180100602.}
\address{School of Information Technology, Deakin University}
\email{\texttt{julien.ugon@deakin.edu.au}}
\author{David Yost}
\address{Centre for Informatics and Applied Optimisation, Federation University Australia}
\email{\texttt{d.yost@federation.edu.au}}
\keywords{$k$-skeleton, reconstruction, simple polytope}
\subjclass[2010]{Primary 52B05; Secondary 52B12}

\begin{document}
\maketitle
\begin{abstract}
Blind and Mani, and later Kalai, showed that the face lattice of a simple polytope is determined by its graph, namely its $1$-skeleton. Call a vertex of a $d$-polytope \emph{nonsimple} if the number of edges incident to it is more than $d$.
We show that (1) the face lattice of any $d$-polytope with at most two nonsimple vertices is determined by its $1$-skeleton; (2) the face lattice of any $d$-polytope with at most $d-2$ nonsimple vertices is determined by its $2$-skeleton; and (3) for any $d>3$ there are two $d$-polytopes with $d-1$ nonsimple vertices, isomorphic $(d-3)$-skeleta and nonisomorphic face lattices. In particular, the result (1) is best possible for $4$-polytopes.
\end{abstract}

\section{Introduction}
\label{sec:intro}
We say that a $d$-polytope $P$ is \emph{reconstructible} from its $k$-skeleton if the restriction of its 
face lattice to the faces of dimension
at most $k$ determines the entire face lattice of $P$. It easily follows from a generalisation of Jordan's separation theorem\footnote{Every subset of the $d$-sphere which is a homeomorphic image of the $(d-1)$-sphere divides the $d$-sphere into two connected components.} that any $d$-polytope is reconstructible from its $(d-2)$-skeleton~\cite[Thm.~12.3.1]{Gru03}. This is tight: for any $d\geq 4$ Perles found $d$-polytopes which are not combinatorially isomorphic but have isomorphic $(d-3)$-skeleta~\cite[Sec.~12.3]{Gru03}.
Call a vertex in a $d$-polytope \emph{nonsimple} if the number of edges incident to it is more than $d$; call it {\it simple} otherwise.
Note that the $d$-bipyramid and the pyramid over the $(d-1)$-bipyramid form an example of such a pair with exactly $d$ nonsimple vertices in each.
The two polytopes in \cref{fig:nonrecontructible} correspond to the case $d=4$.

\begin{figure}
\begin{center}
\includegraphics[scale=0.8]{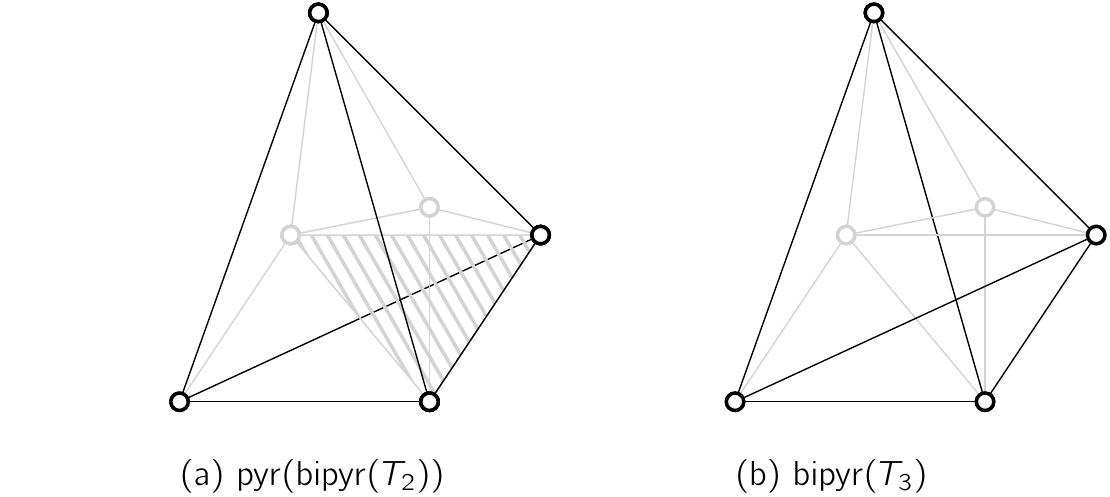}
\end{center}
\caption{A pair of 4-polytopes with four nonsimple vertices which are nonreconstructible from their graphs. The missing 2-face in the bipyramid $\bipyr(T_2)$ is highlighted.}
\label{fig:nonrecontructible}
\end{figure}

For $1\leq k\leq d-3$, let $\beta_{k,d}$ denote the maximum number $j$ such that any $d$-polytope with at most $j$ nonsimple vertices is reconstructible from its $k$-skeleton.
The result of Blind and Mani~\cite{BliMan87}, later proved through a brilliant argument by Kalai~\cite{Kal88}, asserts $0\leq \beta_{1,d}$. Combined with the above example, the following is known for $d\geq 4$:
\[0\leq \beta_{1,d} \leq \beta_{2,d}\leq\ldots\leq \beta_{d-3,d}\leq d-1.\]

We obtain the following result.
\begin{theorem}[Main Theorem]
For any $d\geq 4$,

(1) $\beta_{2,d}= \ldots =\beta_{d-3,d}= d-2$, and

(2) $2\leq \beta_{1,d}\leq d-2$, in particular $\beta_{1,4}=2$.
\end{theorem}

Further, for any fixed $d$ in the above theorem, with exception of the reconstruction from the 1-skeleton of a polytope $P$ with two nonsimple vertices, the reconstruction of the face lattice of the relevant polytopes can be done in polynomial time in the number of vertices.

The proof of $d-2\leq \beta_{2,d}$, based on Kaibel's~\cite[Prop.~1]{Kai03}, is given in \cref{sec:2-skel}.
We give two proofs of $2\leq \beta_{1,d}$ based on a restriction of Kalai's good acyclic orientations~\cite{Kal88} to a subfamily with certain desired properties; see \cref{lem:Orientation-F-Initial}. In addition to this subfamily of orientations, the second proof uses truncation of polytopes to reduce to the easier assertion $1\leq \beta_{1,d}$.
The results on polynomial complexity follow Friedman~\cite{Fri09}.
Pairs of polytopes with $d-1$ nonsimple vertices showing $\beta_{d-3,d}\leq d-2$ are given in \cref{sec:construction}; these are constructed by induction on the dimension and include the pair in \cref{fig:nonrecontructible-3nonsimple}, found in the database by Miyata-Moriyama-Fukuda~\cite{FukMiyMor13}. Realisations of these polytopes are provided via a \texttt{polymake} program, available online at \cite{PinProgram} under the name of the paper.

\begin{figure} 
\begin{center}
\includegraphics{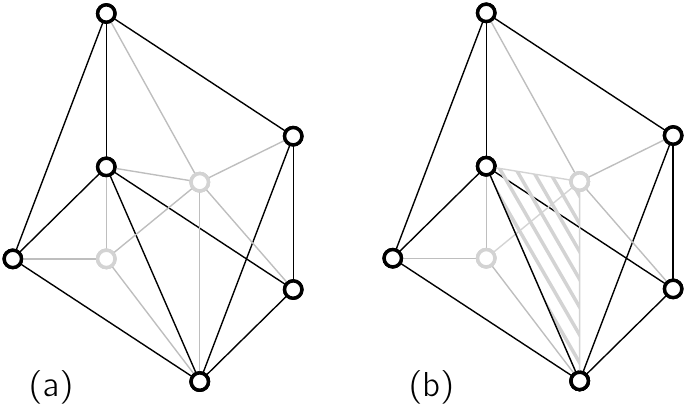}
\end{center}
\caption{A pair of 4-polytopes with three nonsimple vertices which are nonreconstructible from their graphs. The missing 2-face of a bipyramid over a simplex is highlighted in (b), while in (a) this bipyramid is split into two simplices. These polytopes form part of the database of 4-polytopes with 8 vertices by Miyata-Moriyama-Fukuda~\cite{FukMiyMor13}.}
\label{fig:nonrecontructible-3nonsimple}
\end{figure}

We still do not know the answer to the following problem.
\begin{problem}
Does $\beta_{1,d}<\beta_{2,d}$ for some $d\geq 5$?
\end{problem}

\section{Pairs of nonisomorphic $d$-polytopes with $d-1$ nonsimple vertices and isomorphic $(d-3)$-skeleton}
\label{sec:construction}

In this section, for every dimension $d\ge 4$ we construct pairs of nonisomorphic  $d$-polytopes with $d-1$ nonsimple vertices and isomorphic $(d-3)$-skeleta.

First, some terminology, following \cite[p.~241]{Zie95} (for undefined terminology on polytopes see e.g. the textbooks \cite{Gru03, Zie95}). Let $P\subset\mathbb{R}^d$ be a $d$-polytope and let $w$ be a point in  $\mathbb{R}^d\setminus P$. We say that a facet $F$ of $P$ is {\it visible} from the point $w$ with respect to a polytope $P$ in $\mathbb{R}^{d}$ if $w$ belongs to the open halfspace determined by $\aff F$ which is disjoint from $P$. We don't specify $P$ or $ \mathbb{R}^{d}$ when it is clear from the context. If instead $w$ belongs to  the open halfspace which contains the interior of $P$, we say that the facet $F$ is {\it nonvisible} from $w$.  Moreover, the point $w$ is {\it beyond} a face $G$ of $P$ if the facets of $P$ containing $G$ are precisely those that are visible from $w$.

Our construction relies on the following well-known theorem.
\begin{theorem}[{\cite[Thm.~5.2.1]{Gru03}}]
\label{thm:beneath-beyond} Let $P$ and $P'$ be two $d$-polytopes in $\mathbb{R}^d$, and let $v$ be a vertex of $P'$ such that $v\not\in P$ and $P'=\conv (P\cup \{v\})$. Then
\begin{enumerate}[(i)]
\item a face $F$ of $P$ is a face of $P'$ if and only if there exists a facet  of $P$ containing $F$ which is nonvisible from $v$;

\item if $F$ is a face of $P$ then $F':=\conv(F\cup\{v\})$ is a face of $P'$ if

\begin{enumerate}[(a)]
\item either $v\in \aff F$;
\item or among the facets of $P$ containing $F$ there is at least one which is visible from $v$ and at least one which is nonvisible.
\end{enumerate}
\end{enumerate}
Moreover, each face of $P'$ is of exactly one of the above three types. 	
\end{theorem}

\begin{proposition}\label{prop:2-skel-Const}
For every dimension $d\ge4$ there is a pair of $d$-polytopes $Q_d^1$ and $Q_d^2$ with $2d$ vertices, nonisomorphic face lattices and isomorphic $(d-3)$-skeleta, such that each has exactly $d-1$ nonsimple vertices. In particular, $\beta_{d-3,d}\leq d-2$.
\end{proposition}

The proof of the proposition follows from the following two claims.

{\bf Claim 1.} For every $d\ge 3$ there is a $d$-polytope $Q^1_d$ with $2d$ vertices labelled \(0,\ldots,2d-1\) in such a way that (1) the $d-1$ nonsimple vertices of $Q_d^1$  have positive even labels, and that  (2)  its $2d$ facets are as follows. Let \(X\) denote the set of even-labelled vertices except \(0\).
\begin{description}
    \item[Type A] A simplex: \(\{0\} \cup X\);
    \item[Type B] $d-1$ facets of the form: \(\{0\} \cup \{2i+1: i=0,\ldots,k-1\} \cup X\setminus \{2k\}\) for \(k=1,\ldots,d-1\);

    \item[Type C] $d-2$ facets of the form:  \(\{2d-1\} \cup \{2i+1: i=k-1,\ldots,d-2\} \cup X\setminus \{2k\}\) for \(k=1,\ldots,d-2\);
        \item[Type D] A simplex: $\{2d-3,2d-1\}\cup X\setminus \{2(d-1)\}$; and
        \item[Type E] A simplex: \(\{2d-1\} \cup X\).

\end{description}

\begin{proof} The construction of the polytope $Q^1_d$ is by induction, with the base case $d=3$ depicted in \cref{fig:Qd} ~(a). We now construct $Q^1_{d+1}$ from $Q^1_d$.

\begin{figure}
\begin{center}
\includegraphics[scale=1.1]{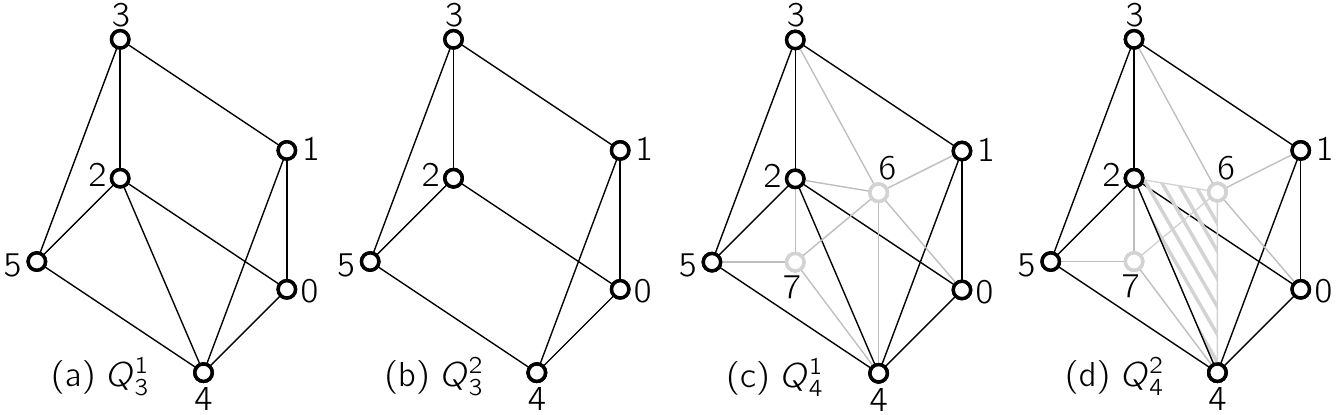}
\end{center}
\caption{The pair of $3$-polytopes $Q^1_3$ and $Q^2_3$ in $(a)-(b)$, and Schlegel diagrams of the pair of $4$-polytopes $Q^1_4$ and $Q^2_4$, projected on the fact isomorphic to $Q^1_3$ in $(c)-(d)$. The missing 2-face of the bipyramid face $02467$ in $Q^2_4$ is highlighted.}
\label{fig:Qd}
\end{figure}

\begin{enumerate}
\item Construct a pyramid over \(Q_d^1\) and label the apex of the pyramid by \(2d\), and let \(X': = X\cup \{2d\}\).
    \item Take the convex hull with a new vertex $v$ labelled \(2d+1\) positioned on the affine hull of the triangle \(\{2d-3,2d-1,2d\}\) in such a way that every facet not containing the edge \([2d-1,2d]\) is nonvisible from the point $v=2d+1$, and {\it the} facet containing the edge  but not the triangle is visible from the point $v$. Note that this triangle is a proper face of the pyramid because of the facet of Type D of \(Q_d^1\).
\end{enumerate}

The facets of the polytope $Q^1_{d+1}$ are as follows.

From \cref{thm:beneath-beyond}(i) it follows that any facet of the pyramid not containing the edge \([2d-1,2d]\) will remain a facet of the polytope. These are our facets of Types A-B.
	\begin{description}
	    \item[Type A] A simplex: \(\{0\} \cup X'\);
	    \item[Type B] $d+1-1$ facets of the form: \(\{0\} \cup \{2i+1: i=0,\ldots,k-1\} \cup X'\setminus \{2k\}\) for \(k=1,\ldots,d+1-1\);
	\end{description}
\cref{thm:beneath-beyond} (ii-a) gives that any facet $F$ of the pyramid containing the triangle \(\{2d-3,2d-1,2d\}\) is contained in the corresponding facet $F'=\conv(F\cup\{2d+1\})$ in the new polytope. These new facets are of Type C.
	\begin{description}
	    \item[Type C] $d+1-2$ facets of the form:  \(\{2(d+1)-1\} \cup \{2i+1: i=k-1,\ldots,d+1-2\} \cup X'\setminus \{2k\}\) for \(k=1,\ldots,d+1-2\);
	\end{description}
Finally, \cref{thm:beneath-beyond} (ii-b) ensures that  the union of the vertex $v=2d+1$ with every $(d-2)$-face not containing the edge \([2d-1,2d]\) of the only remaining facet of the pyramid, which is the simplex \(\{2d\}\cup \{2d-1\} \cup X\), will form a facet. There are exactly two such facets; the Types D-E.
	\begin{description}
\item[Type D] A simplex: $\{2(d+1)-3, 2(d+1)-1\} \cup X'\setminus \{2d\}$; and
	    	    \item[Type E] A simplex: \(\{2(d+1)-1\} \cup X'\);
	\end{description}
It remains to show that the nonsimple vertices of $Q^1_{d+1}$ are exactly the elements of $X'$.
By induction, the nonsimple vertices of $Q^1_{d}$ form the set $X$, so in the pyramid over $Q^1_{d}$ the nonsimple vertices form the set $X'$. Taking the convex hull with $v=2d+1$, the edge $\{2d-1,2d\}$ disappears, and the edges containing $v$ are created, with one of them being $\{v,2d\}$. Thus, it remains to check that $v$ is simple; indeed, $v$ is adjacent  to exactly the vertices in $X'\cup\{2d-1\}$, as \cref{thm:beneath-beyond} (ii-b) shows.
This completes the proof of the claim. 	
\end{proof}	

{\bf Claim 2.} For every dimension $d\ge 4$ there is a polytope $Q^2_d$ with the same $(d-3)$-skeleton as $Q^1_d$, whose $2d$ vertices are labelled \(0,\ldots,2d-1\) such that (1) the $d-1$ nonsimple vertices of $Q_d^2$  have positive even labels, and (2)  the $2d-1$ facets are as follows.
\begin{description}
    \item[Type A'] A bipyramid over a simplex: \(\{0,2d-1\} \cup X\);
    \item[Type B'] $d-1$ facets of the form: \(\{0\} \cup \{2i+1: i=0,\ldots,k-1\} \cup X\setminus \{2k\}\) for \(k=1,\ldots,d-1\);
 
    \item[Type C'] $d-2$ facets of the form:  \(\{2d-1\} \cup \{2i+1: i=k-1,\ldots,d-2\} \cup X\setminus \{2k\}\) for \(k=1,\ldots,d-2\); and
    \item[Type D'] A simplex: $\{2d-3,2d-1\}\cup X\setminus \{2(d-1)\}$.
\end{description}
 
In short, the polytope $Q^2_d$ is created by gluing the simplex facets of Type A and Type E of $Q^1_d$ along the ridge with vertex set $X$ to create a bipyramid of $Q^2_d$, the facet of Type A'. The ridge with vertex set $X$ of $Q^1_d$ then becomes a missing ridge in $Q^2_d$.

\begin{proof}  We construct the $d$-polytope $Q^2_d$ by taking the convex hull of $Q_d^1$ and a new vertex $v^*$. 

First consider the edge $e=[2d-3,2d-1]$ of $Q_d^1$ and the unique facet $F$ containing the vertex $2d-1$ but not $2d-3$; note that the vertex $2d-1$ is simple. The facet $F$ is a simplex with vertices $\{2d-1\}\cup X$. Place the vertex $v^*$ beyond the facet $F$ along the the ray emanating from the point $2d-3$ and containing the edge $[2d-3,2d-1]$ so that $v^{*}$ lies on the first hyperplane $H$ encountered which supports some facet $\bar F$. This ensures that any facet of  $Q_d^1$ different from $F$, $\bar F$ or the $d-1$ facets $F^e_1,\ldots,F^e_{d-1}$ containing the edge $e$ is nonvisible from the vertex $v^*$.  To show that the hyperplane $H$ above exists, note that in the construction of $Q^1_d$ we can place the vertex $v$ labeled $2d-1$ arbitrarily close to the vertex $2d-2$, thereby ensuring that the ray emanating from $2d-3$ and containing the edge $e$ intersects a hyperplane which supports a facet containing the vertex $2d- 2$ but not the edge $e$.  Such a facet exists; for example, the facet with vertex set $\{0\}\cup X$.The polytope $Q_d^2$ is the convex hull of $Q_d^1$ and $v^*$; thus, the vertex set of  $Q^2_d$  is obtained from the vertex set of  $Q^1_d$ by deleting $v$ and adding $v^{*}$, which we also label as $2d-1$.

 Since $v^*\in H$, there is at least one ridge $R$ of $\bar F$ which is visible from $v^*$ with respect to $\bar F$ in $H$. This implies that the other facet  containing $R$ is visible from $v^{*}$ with respect to $Q^1_d$ in $\mathbb{R}^{d}$. Since there is exactly one facet of $Q^1_d$ visible from $v^*$ (with respect to $Q^1_d$), namely, $F$,   we must have $R=F\cap \bar F $ in $Q_d^1$ and $R$ is unique.

In addition to $\bar F$ not containing the edge $[2d-3,2d-1]$, it does not contain the vertex $2d-1$ either. This implies that the set of vertices of $R$ is $X$. In particular, the facet $\bar F$ must be the facet with vertex set $\{0\}\cup X$ of $Q_d^1$.

From \cref{thm:beneath-beyond} (ii-a) it follows that the facet $\bar F$ of $Q^1_d$ is replaced by  the facet with vertex set $\{0,v^*\}\cup X$, which is a bipyramid over the simplex $R$ with vertex set $X$. Combinatorially, with the exception of the ridge $R$, this bipyramid has the same face lattice as the boundaries of the union of the two simplex facets $F$ and $\bar F$ (Types A and E) of $Q_d^1$.

Now consider any face $J$ of $Q^1_d$ not contained in $\bar F$. We have three possibilities: (1) the face is contained in a facet nonvisible from the vertex $v^*$, (2) the face contains the edge $e=[2d-3,2d-1]$,  and (3) the face does not contain the edge and it is not contained in a facet nonvisible from the vertex $v^*$. In the first case, by \cref{thm:beneath-beyond} (i), this face is also a face of $Q^2_d$. In the second case, the vertex $v^*$ is in the affine hull of the face and by \cref{thm:beneath-beyond} (ii-a), the corresponding face $J'$ in $Q^2_d$ has the same dimension as $J$ and the form $\conv (J\cup \{v^*\})$. In the third case, the face $J$ must be contained in intersections involving the facet $F$ and some facets in $\{F_1^e,\ldots,F_{d-1}^e\}$.  In this case  \cref{thm:beneath-beyond} (i) assures us that $J$ is not a face of $Q^2_d$. In summary,  the two simplex facets $F$ and $\bar F$  in  $Q^1_d$ are replaced by a bipyramid  in $Q^2_d$ with the same $(d-3)$-skeleton as $F\cup \bar F$, and every other facet of $Q^1_d$ falls into the first or second cases:  a facet of $Q^1_d$ falling in the first case remains a facet of $Q^2_d$ and  a facet $J$ of $Q^1_d$ falling into the second case is replaced by the facet $J'=J\cup \{v^*\}\setminus \{v\}$ in $Q^2_d$.

 Note that $X$ remains the set of nonsimple vertices in $Q^2_d$ as well.
This completes the proof of the claim.\end{proof}

Refer to \cref{fig:Qd} ~(c)-(d) for Schlegel diagrams of the polytopes $Q^1_4$ and $Q^2_4$, where the projection facet is isomorphic to $Q_3^1$.

\texttt{polymake} script \cite{GawJos00} implementing the ideas presented in \cref{prop:2-skel-Const} is available online at \cite{PinProgram}. Refer to the \texttt{polymake} script for information on how to run the program. Note that, in addition to the polytopes $Q^{1}_{d}$ and $Q^{2}_{d}$, this program constructs other pairs of polytopes with $d-1$ nonsimple vertices and the same $k$-skeleton for a given $k\ge1$ based on the ideas put forward in \cref{prop:2-skel-Const}.

\section{Reconstruction from 2-skeletons}
\label{sec:2-skel}

\begin{theorem}\label{thm:2-skel}
 For any fixed $d\geq 3$,
 let $P$ be a $d$-polytope with at most $\nu$ nonsimple vertices in each facet.
If $\nu\le d-2$ then $P$ is reconstructible from its 2-skeleton.
    Furthermore,  all the facets can be found in linear time in the number of vertices of the graph.
\end{theorem}
This result is best possible as the pair $Q^1_d$ and $Q^2_d$ constructed in \cref{sec:construction} is an example of $d$-polytopes with $\nu=d-1$ and with isomorphic $(d-3)$-skeleton, but with a different number of facets.

For the proof we need the notion of \emph{frames}, and a useful observation of Kaibel about them (cf.~\cite[Prop.~1]{Kai03}), to be spelt out in \cref{prop:Kaibel}.
Define a {\it $k$-frame} as a subgraph of $G(P)$ isomorphic to the star $K_{1,k}$, where the vertex of degree $k$ is called the {\it root} or {\it centre}  of the frame. If the root of a frame is a simple vertex of $P$, we say that the frame is {\it simple}. For a simple vertex $v$ in a $k$-face $F$ of a polytope $P$, we say that the $k$-frame $t_v$ \emph{defines} $F$ if $t_{v}$ is the unique $k$-frame  with root $v$ that is contained in $F$.

We next rephrase  \cite[Prop.~1]{Kai03} to suit our needs and provide a proof.
\begin{proposition} Let $uv$ be an edge of a $d$-polytope $P$ with $u$ and $v$ being simple vertices in $P$. Let $F$ be a facet of $P$ containing both $u$ and $v$, with $t_{u}$ being the frame  centred at $u$ which defines $F$. If $u'$ is the unique neighbour of $u$ not in $t_{u}$ and $v'$ is the neighbour of $v$, other than $u$, which is contained in the 2-face of $P$ defined by the 2-frame $(u,u',v)$ with root $u$, then $v'$ is not in $F$. 
\label{prop:Kaibel}
\end{proposition} 
\begin{proof} Suppose, by way of contradiction, that $v'$ is in $F$. Denote by $t_{v}$ the $(d-1)$-frame of $v$ defining $F$.  Let  $W$ be the 2-face of $P$ defined by the 2-frame $(u,u',v)$ with root $u$. That is, $v'$ is in $t_{v}$ and in $W$. Since $v$ is in $F$ and every vertex in $t_{v}$ is in $F$, the 2-face $W$, which is also defined by the 2-frame $(v,u,v')$ with root $v$, would be contained in $F$, a contradiction.
 \end{proof}

\begin{proof}[Proof of Theorem~\ref{thm:2-skel}]

We show that a  modification of Friedman's algorithm \cite[Sec.~7]{Fri09} gives a proof of the theorem. Assume that we are given the 2-skeleton of $P$. Repeat the following routine, until all simple $(d-1)$-frames in $G(P)$ are visited.
\begin{enumerate}
\item Pick a simple vertex $u$ (it exists) and select any simple $(d-1)$-frame $t_u$ centred at $u$. Let $u'$ be the unique vertex adjacent to $u$ which is not in that frame. The frame $t_u$ is contained in a unique facet in $P$, denote it by $F_u$.

    \item Consider any other simple vertex $u''$ in the frame $t_u$ with $u''\ne u$ (it exists). Then there exists another simple $(d-1)$-frame $t_{u''}$ centered at $u''$ in the facet $F_u$.

\item Consider the neighbour $\hat u$ of $u''$, different from $u$, which is present in the 2-face $W$ that contains the 2-frame $(u,u',u'')$ with root $u$. We know all the  vertices of $W$. Then, applying \cref{prop:Kaibel} to the edge $uu''$ gives that the frame $t_{u''}$ is formed by all the vertices adjacent to $u''$ other than $\hat u$. 
\item Continue this process, always moving along edges formed by simple vertices, and stop when no new simple $(d-1)$-frame can be visited.
\end{enumerate} 
We show that when (4) stops the obtained graph spans the facet $F_u$. The conditions of the theorem guarantee that, in any facet $F_u$ of $P$, we can go from any simple vertex to any other simple vertex through a path only formed of simple vertices, since $G(F_u)$ is $(d-1)$-connected by Balinski's Theorem \cite[Sec.~3.5]{Zie95}. Thus, after Step (4) finishes, the vertex set of the graph obtained is the vertex set of the facet $F_u$.

Repeating this process for all simple $(d-1)$-frames will then reveal all the facets of $P$.

Finally, we show that this process, with a little preprocessing, runs in linear time in the number of vertices.
The number of simple $(d-1)$-frames is $d$ times the number of simple vertices of $P$, thus linear, and each such frame is visited once. It remains to check that the move from one simple $(d-1)$-frame, $t_u$, to the next, $t_{u''}$, can be done in constant time (depending on $d$): checking if the degree of a vertex (neighbour of $u$) is $d$ takes only constant time, and in case it is $d$ we need to find $\hat{u}$ in step (3) in constant time. For this, preprocess the data of the $2$-skeleton, to construct the graph $G_T$ (following Friedman~\cite{Fri09}) whose vertices are the $2$-frames, and two of them are connected by an edge iff the root of one is a vertex of the other and both belong to the same $2$-face. Move along edges of $G_T$, from the $2$-frame $(u,u',u'')$ with root $u$ to $(u'',u,\hat{u})$ with root $u''$, to find $\hat{u}$ in constant time.
As the number of faces in the $2$-skeleton is linear in the number of vertices (cf.~\cref{rem:linear-2-faces}), constructing $G_T$ takes linear time.
Thus, the complexity result follows.
\end{proof}
\begin{remark}\label{rem:linear-2-faces} For $k\in[1,d-1]$ the number of $k$-faces in a $d$-polytope with at most $f(d)$ nonsimple vertices is linear in the number $f_{0}$ of vertices of the polytope. Here $f$ is a function in $d$ independent of $f_{0}$. To see this, note that the number of $k$-faces involving a simple vertex is  at most $f_0 \binom{d}{k}$, and the number of $k$-faces all whose vertices are nonsimple is at most  $\binom{f(d)}{k+1}$; the assertion follows.
\end{remark}
	 
\begin{remark}
Let $P$ be a $d$-polytope with exactly $d-1$ nonsimple vertices, denote their set by $N$. Apply the above algorithm to obtain a collection of graphs $G_i$. The vertex sets $V(G_i)$ are exactly the vertex sets of the facets of $P$, providing reconstruction, unless the following happens: there is a single facet $F$ which contains $N$ and $N$ separates $G(F)$, in which case there are two subgraphs $G$ and $G'$ obtained by the algorithm such that $G\cup G'$ spans $F$ and $V(G\cap G')=N$. The only case where we cannot reconstruct is in case the induced graph $G(P)[N]$ on $N$ is complete, and there is ambiguity whether $P$ has two facets corresponding to $V(G)$ and $V(G')$ intersecting on a common ridge with vertex set $N$, or $P$ has a facet with vertex set $V(G\cup G')$ (all other facets are determined); this is demonstrated in the constructions of \cref{sec:construction}. Thus, if the \emph{parity} of the number of facets is also given, we can reconstruct.
\end{remark}

Combining Proposition~\ref{prop:2-skel-Const} and Theorem~\ref{thm:2-skel} gives
\begin{corollary}\label{cor:2-skel}
$\beta_{2,d}= \ldots =\beta_{d-3,d}= d-2$.
\end{corollary}
	
\section{Reconstruction from graphs}
\label{sec:1-skel}
 In this section we prove that, like simple polytopes, $d$-polytopes with at most two nonsimple vertices are reconstructible from their graphs; see \cref{thm:expRecTwo}. This is best possible for $d=4$, as the 4-polytopes $Q^1_4$ and $Q^2_4$ in \cref{fig:nonrecontructible-3nonsimple} have three nonsimple vertices and the same graph. In case of one nonsimple vertex the reconstruction can be done in polynomial time, following Friedman~\cite{Fri09}. 
 
We start with some preparations.

\subsection{Special good orientations}
 
Following Kalai \cite{Kal88}, call an acyclic orientation  of $G(P)$ {\it good } if for every nonempty face $F$ of $P$ the graph $G(F)$ of $F$ has a unique sink\footnote{ A {\it sink} is a vertex with no directed edges going out.}. Actually, we only need that the acyclic orientation has a unique sink in every facet, so for us this possibly larger set represents the good orientations. 
The following remark is simple but important.

\begin{remark}\label{rmk:acycOrient} Let $P$ be a polytope, $O$ an acyclic orientation of $G(P)$, $F$ a $k$-face of $P$ with $k\ge 2$, and let $w$ be any vertex in $G(F)$. Then there is a directed path in $G(F)$ from $w$ to some sink in $F$ and a directed path in $G(F)$ from some source\footnote{ A {\it source} is a vertex with no directed edges coming in.} in $F$ to $w$.
\end{remark}

Define an {\it initial} set with respect to some orientation as a set such that no edge is directed from a vertex not in the set to a vertex in the set. Similarly, a {\it final} set with respect to some orientation is a set such that no edge is directed from a vertex in the set to a vertex not in the set. 

We proceed with a remark where initial sets play an important role.

\begin{remark}
\label{rmk:orientation-composition} Let $P$ be a $d$-polytope, let $F$ be a face of $P$ and let $O$ be a good orientation of $G(P)$ in which $V(F)$ is initial. Further, denote by $O|_F$ the good orientation of $G(F)$ induced by $O$.  If $O_{F}'$ is a good orientation of $G(F)$ other than $O|_F$, then the orientation $O'$ of $G(P)$ obtained from $O$ by directing  the edges of $G(F)$ according to $O_{F}'$ is a also good orientation.
\end{remark}

The next lemma establishes the existence of good orientations with some special properties.

\begin{lemma} \label{lem:Orientation-F-Initial} Let $P$ be a polytope. For every two disjoint faces $F_{i}$  and $F_{j} $ of $P$,  there is a good orientation of $G(P)$ such that (1) the vertices in $F_{i}$ are initial, (2) the vertices in $F_{j}$ are final, and (3) within the face $F_{i}$, any two vertices (if they exist) can be chosen to be the (local) sink and the (global) source.
\end{lemma}

\begin{proof}  We first preprocess the given $d$-polytope $P$ to obtain a projectively equivalent polytope $P'$ in which the supporting hyperplanes of the faces $F_{i}$ and $F_{j}$ are parallel.  We provide the relevant transformation next.

Embed $P$ in a hyperplane $H_{\text{emb}}$ of $\mathbb{R}^{d+1}$ not passing through the origin, say $H_{\text{emb}}:=\{\vec x\in \mathbb{R}^{d+1}: \text{$x_{d+1}=1$}\}$. 
 Within $H_{\text{emb}}$ consider affine $(d-1)$-spaces $K_{i}$ and $K_{j}$ supporting  the faces $F_{i}$ and $F_{j}$ of $P$, respectively.   The intersection of $K_{i}$ and $K_{j}$ in $H_{\text{emb}}$ is an affine $(d-2)$-space which is disjoint from $P$. Consider a hyperplane $H_{\infty}$ in $\mathbb{R}^{d+1}$ through the origin whose intersection with $H_{\text{emb}}$ contains $K_{i}\cap K_{j}$ and whose positive half-space $H_{\infty}^{+}$ contains $P$ in its interior. Finally, let $H_{\text{proj}}$ be any hyperplane in $H^{+}_{\infty}$ parallel to $H_{\infty}$ and disjoint from $P$ but not coinciding with $H_{\infty}$. Following Ziegler \cite[Sec.~2.6]{Zie95}, the hyperplane $H_{\text{proj}}$ is called {\it admissible} for $P$. The projective $d$-space can be thought of as the union $H_{\text{proj}}\cup H_{\infty}$, where $H_{\infty}$ collects the points at $\infty$. We map $H_{\text{emb}}\setminus H_{\infty}$ onto $H_{\text{proj}}$ by  sending each point in $H_{\text{emb}}\setminus H_{\text{proj}}$ to the point in $H_{\text{proj}}$ lying on the same line through the origin, while the points in $H_{\text{emb}}\cap H_{\text{proj}}$ remain fixed. The image of $P$ under this map is a polytope $P'$  in $H_{\text{proj}}$ which is combinatorially equivalent to $P$. Since the intersection of the spaces $K_{i}$ and $K_{j}$ lies in $H_{\infty}$, their projections on $H_{\text{proj}}$ are parallel. This completes this transformation. We may therefore assume that $P$ has undergone the transformation we have just described.

Now back in $\mathbb{R}^{d}$ consider a hyperplane $K_{i}$ which supports the face $F_{i}$ of $P$ and is parallel to a hyperplane supporting $F_{j}$.  Let $g$ be a linear function which vanishes on $K_{i}$ and whose value on $K_{j}$ is positive. Perturb $g$ slightly so that the resulting linear function $f$ attains different values on the vertices of $P$. The function $f$ ensures the existence of a good orientation $O$ in which the vertices in $F_{i}$ are initial while the vertices in $F_{j}$ are final; this proves the conditions (1) and (2). 

To get the condition (3), consider the polytope $F_{i}$ in $\aff F_{i}$ (forgetting about $P$), and let $s$ and $t$ be two arbitrary vertices in $F_{i}$, if they exist. Performing the aforementioned projective transformation to $F_{i}$ in $\aff F_{i}$, we can assume that the vertices $s$ and $t$ admit parallel supporting hyperplanes in the space $\aff F_{i}$. Reasoning as before gives a good orientation $O_{i}$ of $F_{i}$ in which $s$ is a sink and $t$ is a source. If, in the orientation $O$, we reorient the edges of $F_{i}$ according to $O_{i}$, the resulting orientation $O'$ of $G(P)$ remains good since $F_{i}$ is initial (c.f.~\cref{rmk:orientation-composition}), thereby satisfying all the three conditions. This completes the proof of the lemma.        
 \end{proof}

Any induced $(d-1)$-connected subgraph of $G(P)$ where simple vertices in the polytope have each degree $d-1$ and nonsimple vertices have each degree $\ge d-1$ is called a {\it feasible}  subgraph. We say that a $k$-frame with root $x$ is {\it valid}  if there is a facet of $P$ containing $x$ and the edges of the frame and no other edge incident to $x$. Thus, if $x$ is a simple vertex then any of its $(d-1)$-frames is valid.

\begin{lemma}\label{lem:feasible-subgraphs}  Let $P$ be a $d$-polytope, and let $H$ be a feasible subgraph of $G(P)$ containing at most $d-2$ nonsimple vertices. If the graph $G(F)$ of some facet $F$ is contained in $H$, then $H=G(F)$.
\end{lemma}

\begin{proof}
If $\ver F = V(H)$, then $G(F)= H$, as $G(F)$ is an induced subgraph of $G(P)$. Otherwise, $\ver F \subsetneq V(H)$, in which case any path from a vertex in $H\setminus G(F)$ to a vertex in $G(F)$ must pass through a nonsimple vertex, since simple vertices of the polytope have the same degree in both $H$ and $G(F)$. Consequently, the nonsimple vertices would disconnect $H$, contradicting its $(d-1)$-connectivity.

\end{proof}	

\subsection{Polytopes with one nonsimple vertex}\label{subSec:one-two-vertices}
\begin{theorem}\label{thm:1vertex}
 Let $d\ge 3$. Every $d$-polytope with at most one nonsimple vertex can be reconstructed from its graph, in polynomial time in the number of vertices.
\end{theorem}

The proof is an adaptation of proofs from \cite{JosKaiKor02} and \cite{Fri09} to the case when one nonsimple vertex exists. As such we only provide a sketch with the main ingredients.
\begin{proof}[Sketch of proof of Theorem~\ref{thm:1vertex}]
First, we consider only the set $\mathcal{A}$ of acyclic orientations of the polytope graph in which the nonsimple vertex, if present, has indegree 0. The existence of good orientations in $\mathcal{A}$ follows from \cref{lem:Orientation-F-Initial}. Second, we need a slight generalisation of the $2$-systems of \cite{JosKaiKor02}. A set $S$ of subsets of $\ver P$ is called a  {\it $2$-system} of $G(P)$ if for every set $S'\in S$ the subgraph induced by $S'$ is $2$-regular and if the vertex set of every $2$-frame of $P$ with a simple vertex as a root is contained in a unique set of $S$. Notice that the set $V_2(P)$ of vertex sets of $2$-faces of $P$ is a $2$-system of $G(P)$. Then a result, namely \cref{lem:Thm-2systems}, in the same vein as \cite[Thm.~1]{JosKaiKor02} and \cite[Thm.~4]{Kai03} can be obtained, following their proofs.
  
\begin{lemma}\label{lem:Thm-2systems}
Let  $P$ be a $d$-polytope with at most one nonsimple vertex, $S$ a 2-system of $G(P)$ and $O$ an acyclic orientation of $\mathcal A$. Then, as argued there,
\[|S|\le f_2(P)\le f_2^O:=\sum_{i=0}^d h_i^O {i \choose 2},\]
where $h_k^O$ denotes the number of vertices of $G$ with indegree $k$. The first inequality holds with equality iff $S=V_2(P)$, and the second inequality holds with equality iff $O$ is a good orientation of $\mathcal{A}$.
\end{lemma}

Next we present all the relevant programs to compute $V_2(P)$, which are those presented in \cite[Sec.~4]{Fri09}, with some changes.  

Let $t$ denote the 2-frame $(t_0,t_1,t_2)$ with root $t_0$ and let $T$ denote the set of all 2-frames in $G$ in which the nonsimple vertex $v$ is not a root. Let $W$ be the set of 2-regular induced subgraphs in $G$. The integer program \text{IP-S} finds a 2-system of maximum cardinality.  
\begin{align*}
\begin{alignedat}{4}
&(\text{IP-S}) \quad && \text{max}       && \sum_{w\in W} x_{w} & \\
&              \quad && \text{s.t.} \quad&& \forall t\in T, \sum_{w\in \delta(t)} x_{w}= 1,& \\
&              \quad &&                  &&x_{w}\in \{0,1\}.
\end{alignedat}&&
\begin{alignedat}{4}
&(\text{LP-S}) \quad && \text{max}       && \sum_{w\in W} x_{w} & \\
&              \quad && \text{s.t.} \quad&& \forall t\in T, \sum_{w\in \delta(t)} x_{w}\le 1,& \\
&              \quad &&                  &&x_{w}\ge 0.
\end{alignedat}
\end{align*}
 Here $\delta(t)$ is the set of all elements in $W$ containing the 2-frame $t$. As in \cite[Sec.~4]{Fri09}, we allow 2-regular induced subgraphs which are union of cycles, since the maximum will inevitably occur with positive weight only on single cycles. Note that in \text{IP-S} the set $W$ may have exponential size, but the set $T$, and hence the number of equations, has only polynomial size. We relax and dualise \text{IP-S}, obtaining \text{LP-S} and \text{LP-SD}, respectively.
\begin{align*}
\begin{alignedat}{4}
&(\text{LP-SD}) \quad && \text{min}       && \sum_{t\in T} y_{t} & \\
&              \quad && \text{s.t.} \quad&& \forall w\in W, \sum_{t\in w} y_{t}\ge 1,& \\
&              \quad &&                  &&y_{t}\ge 0.
\end{alignedat}&&
\begin{alignedat}{4}
&(\text{IP-$f_2$}) \quad && \text{min}       && \sum_{t\in T} y_{t} & \\
&              \quad && \text{s.t.} \quad&& \forall w\in W, \sum_{t\in w} y_{t}\ge 1,& \\
&              \quad &&                  &&y_t\in\{0,1\},\\
&              \quad &&                  &&\text{$y_{t}$ arises from $O\in \mathcal A$}.
\end{alignedat}
\end{align*}

 Let \text{IP-SD} be the related binary-integer program for \text{LP-SD}: replace $y_t\ge 0$ with $y_t\in\{0,1\}$. 

Consider an acyclic orientation $O\in \mathcal{A}$ and let $y_t=1$ represent the case where the root of the 2-frame $t$ is a sink of $t$. Then, according to \cref{lem:Thm-2systems}, the integer program \text{IP-$f_2$} finds an $\aof$-orientation $O\in \mathcal{A}$ of $G$.

As in \cite[Thm.~2]{Fri09}, applying \cref{lem:Thm-2systems} and the strong duality theorem of linear programming to this sequence of optimisation problems gives the following similar result, with the same proof.

\begin{lemma}\label{lem:optProb}
Let $P$ be a $d$-polytope with at most one nonsimple vertex and let $G$ denote its graph. Then the aforementioned optimisation problems {\text IP-S}, \text{ LP-S}, {\text LP-SD}, {\text IP-SD} and {\text IP-$f_2$} all have the same optimal value.
\end{lemma}

We proceed by solving {\text LP-SD}. The program {\text LP-SD} has a polynomial number of variables, an exponential number of constraints and  all the constraints with polynomially bounded size.  As in \cite[Sec.~5]{Fri09},  this problem can be solved using the ellipsoid method. An important feature of the ellipsoid method is that it is not necessary to have an explicit list of all inequalities ready at hand. It suffices to have a ``separation oracle'' which, given a vector $\vec y$, decides whether or not $\vec y$ is a solution of the system. If $\vec y$ is a solution, it returns ``yes'', otherwise it returns one (arbitrary) inequality of the system that is violated by $\vec y$, that is, an inequality which separates $\vec y$ from the solution set. Furthermore, if the separation oracle runs in polynomial time then so does the ellipsoid method. 

Our separation algorithm reduces to that of Friedman's \cite[Sec.~5]{Fri09}   after we  produce a new linear program \text{LP-SD-A} equivalent to \text{LP-SD}. The new program \text{LP-SD-A} considers the set $T'$ of all 2-frames in $G$, not only those in which a simple vertex is a root. It adds a new variable $y_t$  and a new constraint $y_t=0$ for every 2-frame $t$ with the nonsimple vertex $v$ as the root. 

\begin{alignat*}{4}
&(\text{LP-SD-A}) \quad && \text{min}       && \sum_{t\in T'} y_{t} & \\
&              \quad && \text{s.t.} \quad&& \forall w\in W, \sum_{t\in w} y_{t}\ge 1,& \\
&              \quad &&                  &&\forall t\in T' \text{such that $t$ has $v$ as a root,}  \;  y_{t}=0,\\
&              \quad &&                  &&y_{t}\ge 0.
\end{alignat*}
  
It is not difficult to see that  \text{LP-SD-A} is equivalent to \text{LP-SD}; that is, any feasible solution of \text{LP-SD-A} corresponds to a feasible solution of \text{LP-SD}, and vice versa.

With a separation algorithm running in polynomial time at hand, we have that the program {\text LP-SD}, and thus, the programs {\text LP-S} and {\text IP-S}, can be solved in polynomial time.  

It would only remain to show that the solution obtained by the program \text{LP-S} is unique and thus a solution vector $\vec x^*$ corresponds to the incidence vector of the 2-faces of the polytope.   The proof of this fact proceeds mutatis mutandis as in the proof of \cite[Thm.~4]{Fri09}.

Finally, with the 2-faces available, we can reconstruct the vertex-facet incidences of the polytope using \cref{thm:2-skel} in linear time. \end{proof}

\subsection{Polytopes with two nonsimple vertices}\label{subSec:one-two-vertices}

\begin{theorem}\label{thm:expRecTwo} If $P$ is a $d$-polytope with two nonsimple vertices, then the graph of $P$ determines the entire combinatorial structure of $P$.
\end{theorem}
In this case the reconstruction algorithms we suggest run only in exponential time. Here is our first proof.
\begin{proof}[Proof of Theorem~\ref{thm:expRecTwo}]
Let us assume that $d\ge4$, as the result is trivial for smaller $d$. Denote by $u$ and $v$ the nonsimple vertices of $P$.
Partition the facets of $P$ into four families: let $\mathcal{F}_{u-v}$ (resp. $\mathcal{F}_{v-u}$) denote the family of facets containing $u$ and not $v$ (resp. $v$ and not $u$) and $\mathcal{F}_{\emptyset}$ (resp. $\mathcal{F}_{vu}$) the family of facets containing none of (resp. both) $u$ and $v$. We find these families in the order $\mathcal{F}_{u-v}$, $\mathcal{F}_{v-u}$, $\mathcal{F}_{\emptyset}$, $\mathcal{F}_{vu}$, from first to last, as given in the following four Claims.

Denote by $\mathcal{H}_u$ the set of feasible subgraphs of $G(P)$ which contain $u$ but not $v$. Denote by  $\mathcal{A}_u$  the set of all acyclic orientations of $G(P)$ in which (1) the nonsimple vertex $u$ has indegree 0, (2) the nonsimple vertex $v$ has outdegree 0, and (3) some subgraph $H_u$ in $\mathcal{H}_u$ is initial. It follows that $H_u$ has a sink which is a simple vertex.

{\bf Claim 1 (find $\mathcal{F}_{u-v}$).} A feasible subgraph $H_u$ in $\mathcal{H}_u$  is the graph of a facet of $P$ containing $u$ but not $v$ iff (1) $H_u$ is initial with respect to a good orientation $O$ in $\mathcal{A}_u$, and (2) $H_u$ has a unique sink which is a simple vertex.

\begin{proof}  First consider a facet $F_u$ containing $u$ but not $v$ (such facet clearly exists). Applying \cref{lem:Orientation-F-Initial} to the faces $F_{u}$ and $v$ we get a good orientation of $G(P)$ in which the vertices of $F_u$ are initial, $u$ is the global minimum and $v$ is the global maximum. Under this orientation, taking $H_u=G(F_u)$ ensures that $\mathcal{A}_u$ is nonempty.

We prove the converse. Let $O\in \mathcal{A}_u$ and let $h_k^O$ denote the number of simple vertices of $G$ with indegree $k$ w.r.t. $O$. Define\[f^O_{u}:=h^O_{d-1}+dh_d^O.\]
The function $f^O_{u}$ counts the number of pairs $(F,w)$, where $F$ is a facet of $P$ and $w$ is a simple sink in $F$ w.r.t. the orientation $O$ in $\mathcal{A}_u$. If $w$ is a simple vertex in $P$ with indegree $k$, then $w$ is a sink in ${k\choose d-1}$ facets of $P$. Since the orientation is acyclic, every facet has a sink. Furthermore, every facet containing $v$ has $v$ as a sink, and $u$ is not a sink in any facet.

Let $H_u\in \mathcal{H}_u$, and let $x$ be the simple sink in $H_u$ with respect to $O$. Suppose $H_u$ does not represent  the facet $F$ of $P$ containing $x$ and the $d-1$   edges in $H_{u}$ incident to $x$, namely $H_u\neq G(F)$. Then, by \cref{lem:feasible-subgraphs}, there is a vertex of $F$ not in $H_u$. Since $H_u$ is initial with respect to $O$, the facet $F$ would contain two sinks, one of them being $x$.
Consequently, as there is a good orientation in $\mathcal{A}_u$ and a subgraph in $\mathcal{H}_u$ representing a facet, we have that \[\min_{O\in \mathcal{A}_u} f^O_{u}=f_{d-1}-f_{d-1}^{v},\] where $f_{d-1}$  (resp.~ $f_{d-1}^{v}$) denotes the number of facets in $P$ (resp. containing $v$).  Also, the orientations in $\mathcal{A}_u$ minimising $f_{u}^O$ are exactly the good orientations in $\mathcal{A}_u$.

Let $x$ be the simple sink in $H_u$ with respect to a good orientation $O$ in $\mathcal{A}_u$. Then $x$ defines a unique facet $F$ of $P$, and all the other vertices of $F$ are smaller than $x$ with respect to the ordering induced by $O$. Since $H_u$ is an initial set in $O$ and since there is a directed path in $G(F)$ from any vertex of $G(F)$ to $x$ (cf.~\cref{rmk:acycOrient}),  we must have $\ver F \subseteq V(H_u)$, and we are done by \cref{lem:feasible-subgraphs}. \end{proof}

Exchanging the roles of $u$ and $v$, define $\mathcal{H}_v$ and $\mathcal{A}_v$ similarly to $\mathcal{H}_u$ and $\mathcal{A}_u$. By symmetry we also get the following claim.

{\bf Claim 2 (find $\mathcal{F}_{v-u}$).} A feasible subgraph $H_v\in \mathcal{H}_v$  is the graph of a facet of $P$ containing $v$ but not $u$ iff (1) $H_v$ is initial with respect to some good orientation $O\in \mathcal{A}_v$ and (2) has a unique sink which is a simple vertex.

Running through all the good orientations in  $\mathcal{A}_u$ or $\mathcal{A}_v$, we recognise all the graphs of facets in $\mathcal{F}_{u-v}$ and in $\mathcal{F}_{v-u}$.
Let $f_{d-1}^{u-v}:=  |\mathcal{F}_{u-v}|$, $f_{d-1}^{v-u}:=  |\mathcal{F}_{v-u}|$, $f_{d-1}^{\emptyset}:=  |\mathcal{F}_{\emptyset}|$ and $f_{d-1}^{uv}:=  |\mathcal{F}_{uv}|$.
Clearly,
\begin{equation}
\label{eq:fd-1}
f_{d-1}=f_{d-1}^{u-v}+f_{d-1}^{v-u}+f_{d-1}^{vu}+f_{d-1}^{\emptyset}.
\end{equation}

Since the number $f_{d-1}-f_{d-1}^{v}$ is known (it is the minimum of $f^{O}_{u}$ over $O\in \mathcal{A}_u$), and since $f_{d-1}^{v} =f_{d-1}^{vu} +f_{d-1}^{v-u}$, it follows that the number $f_{d-1}^{\emptyset}$ is known; if $f_{d-1}^{\emptyset}=0$ then $\mathcal{F}_{\emptyset}=\emptyset$.

Assume then $f_{d-1}^{\emptyset}>0$. We show next how to recognise the facets in $\mathcal{F}_{\emptyset}$. Denote by $\mathcal{H}_\emptyset$ the set of feasible subgraphs of $G(P)$ which contain neither $u$ nor $v$, and by  $\mathcal{A}_\emptyset$  the set of all acyclic orientations of $G(P)$ in which (1)  the nonsimple vertex $v$ has outdegree 0, and (2) some subgraph $H_\emptyset$ in $\mathcal{H}_\emptyset$ is initial. It follows that $H_\emptyset$ has a sink which is a simple vertex.

Define an {\it almost good orientation} as an acyclic orientation in which every facet with a simple sink has a unique sink.

{\bf Claim 3 (find $\mathcal{F}_{\emptyset}$).}
Let $f_{d-1}^{\emptyset}>0$.
A feasible subgraph $H_\emptyset \in \mathcal{H_\emptyset}$  is the graph of a facet of $P$ containing neither $u$ nor $v$ iff (1) $H_\emptyset$ is initial with respect to some \emph{almost} good orientation $O\in \mathcal{A}_\emptyset$ and (2) $H_\emptyset$ has a unique sink which is a simple vertex.

 \begin{proof} Consider a facet $F_\emptyset$ containing neither $u$ nor $v$; such facet exists by assumption. 
Applying  \cref{lem:Orientation-F-Initial} to the faces $F_{\emptyset}$ and $v$, we get a good orientation of $G(P)$ in which the vertices of $F_\emptyset$  are initial and $v$ is the global maximum.  This proves the ``only if'' part of the claim.

 Let $O\in \mathcal{A}_\emptyset$, and as before $h_k^O$ denote the number of simple vertices of $G$ with indegree $k$. Let $t^O_{u- v}$ denote the number of valid $(d-1)$-frames of $u$ which are contained in facets which do not contain $v$ and have $u$ as a sink. Since we know all the facets containing $u$ but not $v$, we can compute $t^O_{u- v}$ from $O$. Define\[f^O_{\emptyset}:=h^O_{d-1}+dh_d^O+t^O_{u-v}.\]
The function $f^O_{\emptyset}$ counts the number of pairs $(F,w)$, where  $F$ is a facet of $P$, $w$ is sink of $F$ and either $w$ is simple or $w=u$. Any facet containing $v$ has $v$ as a sink.
Consequently, as there is a good orientation in  $ \mathcal{A}_\emptyset$ and a subgraph $H_\emptyset\in \mathcal{H}_\emptyset$ representing a facet, we have that \[\min_{O\in \mathcal{A}_\emptyset} f^O_{\emptyset}=f_{d-1}-f_{d-1}^{v}.\]
Note that an orientation $O$ of $\mathcal{A}_\emptyset$ minimising $f_{\emptyset}^O$ is not necessarily a good orientation; there may be  a facet in which both $u$ and $v$ are sinks, but facets with a simple sink or facets not containing $v$ must have a unique sink.

 The proof now proceeds mutatis mutandis as in the proof of Claim 1. Let $x$ be the simple sink in $H_{\emptyset}$ with respect to $O$, then $x$ together with the $d-1$   edges in $H_\emptyset$ incident to $x$ define a unique facet $F$ of $P$, where all the other vertices of $F$ are smaller than $x$ with respect to the ordering induced by $O$. Since $H_{\emptyset}$ is an initial set in $O$ and since there is a directed path in $G(F)$ from any other vertex of $G(F)$ to $x$ (cf.~\cref{rmk:acycOrient}),  and we must have $\ver F \subseteq V(H_{\emptyset})$, with the result following from \cref{lem:feasible-subgraphs}.
\end{proof}
Running through all the orientations in  $\mathcal{A}_\emptyset$ minimising $ f^O_{\emptyset}$ we recognise $\mathcal{F}_\emptyset$.

It remains to recognise $\mathcal{F}_{uv}$.
We first find out whether or not the number $f^{vu}_{d-1}=0$.
Note that each facet contains a simple $(d-1)$-frame, and any simple $(d-1)$-frame is contained in a (unique) facet.
Thus, $f^{vu}_{d-1}>0$ iff there exists a simple $(d-1)$-frame not contained in any of the graphs of the facets in $\mathcal{F}_{u-v}\cup\mathcal{F}_{v-u}\cup \mathcal{F}_\emptyset$.

Assume $f_{d-1}^{vu}>0$. Recognising the facets in $\mathcal{F}_{uv}$ is done similarly to the previous  cases. Denote by $\mathcal{H}_{vu}$ the set of feasible subgraphs which contain both $u$ and $v$, and by  $\mathcal{A}_{vu}$  the set of all acyclic orientations of $G(P)$ in which   some subgraph $H_{vu}$ in $\mathcal{H}_{vu}$ is initial with a unique sink which is a simple vertex.

{\bf Claim 4.}
Assume $f_{d-1}^{vu}>0$.
 Then a feasible subgraph $H_{vu}\in \mathcal{H}_{vu}$  is the graph of a facet containing both $u$ and $v$ iff (1) $H_{vu}$ is initial with respect to some good orientation $O$ in $\mathcal{A}_{vu}$, and (2) has a unique sink which is a simple vertex.
\begin{proof}
Applying  \cref{lem:Orientation-F-Initial} to a facet $F_{vu}\in \mathcal{F}_{vu}$, we get a good orientation in which the vertices of $F_{vu}$ are initial and the sink in $F_{vu}$ is a simple vertex. This proves the ``only if'' part.

For the ``if'' part, for any $O\in \mathcal{A}_{uv}$ consider  the function \[f^O_{vu}:=h^O_{d-1}+dh_d^O.\] Its minimum value over $\mathcal{A}_{vu}$ is $f_{d-1}$. Any orientation minimising $f^O_{vu}$ must be good. For any orientation $O$ minimising $f^O_{vu}$, proceed as in the ``if'' part in the proof of Claim 1 to conclude the proof.\end{proof}
Thus, by Claims 1--4 we have found all the facets of $P$ from $G(P)$.
\end{proof}

\begin{corollary}\label{cor:1-skel}
$\beta_{1,4}=2$, and for any $d>4$, $\beta_{1,d}\ge 2$.
\end{corollary}

\begin{problem}
Can the proof of Theorem~\ref{thm:expRecTwo} be modified to reconstruct the $2$-faces in polynomial time, as in the case of one nonsimple vertex (Theorem~\ref{thm:1vertex})? This would give a polynomial time reconstruction algorithm for the facets, in the presence of two nonsimple vertices.
\end{problem}

\subsection{Reconstruction via truncation}
We now present a second proof of Theorem~\ref{thm:expRecTwo}, again suffering an exponential running time, based on truncation of polytopes.

Let $P$ be a polytope with face $F$.  {\it $P$ truncated at $F$} \cite[p.~76]{Bro83} is the polytope $P'$ obtained by intersecting $P$ with a halfspace $H^{+}$ which does not contain the vertices of $F$ and whose interior contains the vertices of $P$ that are not contained in $F$. Let $H$ denote the hyperplane bounding $H^{+}$. The face lattices of $P$ and of $P'$ determine each other; for our purposes, we need the following parts of this statement, collected in a lemma. For a polytope $P$ let $V(P)$ and $E(P)$  denote the sets of its vertices and edges, respectively.

\begin{lemma}\label{lem:truncetion_and_lattices}
Let $P'$ be the $d$-polytope $P$ truncated at a face $F$.
\begin{enumerate}[(a)]
\item The vertices of $P'$ are of two types: the vertices in $V(P)\setminus V(F)$ and a vertex $w_{xy}$ for each edge $xy$ of $P$ with a vertex $x$ in $V(F)$ and a vertex $y$ in $V(P)\setminus V(F)$. 

\item The edges of $P'$ are of three types: (1) the edges $y_{1}y_{2}$ in $P$ with $y_{1},y_{2}\in V(P)\setminus V(F)$, (2) the edges $yw_{xy}$ with $y \in V(P)\setminus V(F)$ and $w_{xy}\in H\cap P$, and (3) the edges $w_{x_{1}y_{1}}w_{x_{2}y_{2}}$ with $x_{1},y_{1}, x_{2}, y_{2}$ contained in a 2-face of $P$. In particular, $G(P)$ and the $2$-faces of $P$ containing at least one vertex from $F$ are enough to determine $G(P')$.

\item The facets of $P'$ are of two types: The facet $H\cap P$ and the ``old'' facets of $P$, except $F$ if it is indeed a facet; that is, the facets $J':=H^{+}\cap J$, where $J$ is a facet of $P$ possibly other than $F$ . Hence, given the vertex set of a facet $J'$ of $P'$ other than $H\cap P$, we obtain the vertex set of the corresponding facet $J$ of $P$ by replacing each vertex $w_{xy}$ in $J'$ with the corresponding vertex $x$ in $F$. Consequently, all facets of $P$ are thus obtained.

\item For any vertex $w_{xy}\in P'$ with $x\in V(F)$ and $y\in V(P)\setminus V(F)$, if $y$ has degree $d$ in $G(P)$ then $w_{xy}$ has degree $d$ in $G(P')$.
\end{enumerate}
\end{lemma}

Let $u$ and $v$ be the two nonsimple vertices of $P$.
Our goal now is to find $G(P')$ where $P'$ is the truncation of $P$ at the edge $uv$ in case $uv\in E(P)$, or the truncation of $P$ at $u$ in case $uv\notin E(P)$. Once we succeed in this goal, we are done by Lemma~\ref{lem:truncetion_and_lattices}(d): in the former case  since $P'$ would be simple, and in the later case since $P'$ would have exactly one nonsimple vertex. So in either case we can reconstruct the facets of $P'$ (in polynomial time). Then by Lemma~\ref{lem:truncetion_and_lattices}(c) we reconstruct the facets of $P$ (again in polynomial time).

By Lemma~\ref{lem:truncetion_and_lattices}(a-b), to achieve this goal it is enough to determine all $2$-faces of $P$ containing at least one of $u$ and $v$ (this we do in exponential time); then we can construct $G(P')$ (in polynomial time).

First we determine the $2$-faces of $P$ containing exactly one of $u$ and $v$: each such $2$-face is contained in a facet containing exactly one of $u$ and $v$; those facets we find, for example, by Claims 1 and 2 from the first proof of Theorem~\ref{thm:expRecTwo}. Then we find the relevant $2$-faces in such facet $T$ by reconstructing the face lattice of $T$ from the  subgraph of $G(P)$ induced by $V(T)$, which has at most one vertex of degree $>d-1$.
Next, we aim to determine the $2$-faces of $P$ containing both $u$ and $v$.

\textbf{Case $uv\in E(P)$}. For any $2$-face $S$ containing $uv$ there is a linear functional $l_S$ that orders $V(P)$ with $u$ first, $v$ second and $S$ initial; to achieve this, start with a linear function that attains its minimum over $V(P)$ exactly at $V(S)$ (cf.~\cref{lem:Orientation-F-Initial}), then perturb it so that it is minimised exactly on the edge $vu$, and finally perturb the resulting linear function again so that it is minimised on $u$ only.  Using the original objective function of Kalai $f^O:=\sum_{w\in V(P)}2^{\text{indeg}_O(w)}$, where $O$ is an acyclic orientation, the functionals $l_S$ show, as in Kalai's proof (see \cite{Kal88} or \cite[Sec.~3.4]{Zie95}), that the vertex sets of $2$-faces of $P$ containing $u$ and $v$ are exactly the vertex sets of induced $2$-regular graphs in $G(P)$ containing $u$ and $v$ which are initial w.r.t. some acyclic orientation $O'$ minimising $f^O$, and such that $\text{indeg}_{O'}(u)=0$ and $\text{indeg}_{O'}(v)=1$. Thus, we can construct $G(P')$, where $P'$ is $P$ truncated at $uv$.
 
\textbf{Case $uv\notin E(P)$}. Then there is at most one $2$-face of $P$ containing both $u$ and $v$. Thus, when constructing $G(P')$, with $P'$ being $P$ truncated at $u$, if we know $G(P)$ and the $2$-faces of $P$ containing $u$ and not $v$, then we may miss at most one edge, one of the form $w_{uy_{1}} w_{uy_{2}}$. However, if we missed such edge, as $w_{uy_{1}}$ and $w_{uy_{2}}$ have degree $d$ in $G(P')$, we would be able to recover that edge: simply connect the unique two vertices of degree $d-1$ in $G(P')$ by an edge. To summarise, we can construct $G(P')$ in this case as well, completing the second proof of Theorem~\ref{thm:expRecTwo}.

\section{Concluding remarks}
\label{sec:conclusion}
In this paper we measured the deviation from being a simple polytope by counting the number of nonsimple vertices, which is perhaps the most natural way. Other measures of such deviation were considered or suggested in the literature.  Blind et al.~\cite{BliBli98} thought of an ``almost'' simple polytope as a $d$-polytope having only vertices of degree $d$ or $d+1$, while in \cite[Sec.~8]{Fri09}, Friedman suggested that $d$-polytopes with few nonsimple vertices of degree at most $d+k$ may be considered close to being simple. In terms of reconstruction, Friedman's and Blind's suggestions are too weak. Perles' construction already gives examples of  polytopes which are not combinatorially isomorphic but share the same $(d-3)$-skeleta, having exactly $d$ vertices of degree $d+1$ while the rest of the vertices are simple.
 
The last three authors considered in ~\cite{PinUgoYos16a, PinUgoYos17} yet another measure of deviation from being a simple polytope, the {\it excess}, defined as $\xi(P):=\sum_v (\deg(v)-d)$, where $\deg(v)$ denote the number of edges incident to the vertex $v$. Simple polytopes have excess zero. The paper \cite{PinUgoYos17} then studied reconstructions of polytopes with small excess and of polytopes with a small number of vertices (at most $2d$).

\section{Acknowledgments}
We thank Micha Perles for helpful discussions and the referees for many valuable comments and suggestions.
Guillermo Pineda would like to thank Michael Joswig for the hospitality at the Technical University of Berlin and for many fruitful discussions on the topics of this research.
Joseph Doolittle would like to thank Margaret Bayer for pushing for more results and keeping the direction of exploration straight.

\end{document}